\documentclass[10pt]{article}

\usepackage{amsmath,amssymb,amsthm,bm}
\usepackage{geometry}
\usepackage{booktabs}
\usepackage{enumitem}
\usepackage{authblk}
\usepackage[colorlinks=true,linkcolor=blue,citecolor=blue,urlcolor=blue]{hyperref}

\newtheorem{theorem}{Theorem}[section]
\newtheorem{lemma}[theorem]{Lemma}
\newtheorem{proposition}[theorem]{Proposition}

\theoremstyle{definition}
\newtheorem{definition}[theorem]{Definition}
\theoremstyle{plain}
\newtheorem{remark}[theorem]{Remark}
\newtheorem{conjecture}[theorem]{Conjecture}

\newcommand{\C}{\mathbb{C}}
\newcommand{\M}{M}
\newcommand{\norm}[1]{\left\lVert #1\right\rVert}
\newcommand{\ip}[2]{\left(#1,#2\right)}

\newcommand{\tr}{\operatorname{tr}}

\title{Proof of Nobori's generalized B\"ottcher--Wenzel inequality conjecture}
\author[1]{Shuo Shi}
\author[1]{Juan Zhang\thanks{Corresponding author. 
E-mail: shishuomath@foxmail.com(Shuo Shi); zhangjuan@xtu.edu.cn (Juan Zhang). 
This work was supported by the National Natural Science Foundation of China (12671450) and  the National Key Research and Development Program of China (2023YFB3001604)}}
\affil[1]{School of Mathematics and Computational Science, Hunan Key Laboratory for Computation and Simulation in Science and Engineering, Key Laboratory of Intelligent Computing and Information Processing of Ministry of Education,
Xiangtan University, Xiangtan, Hunan,  411105, China}
\date{}

\begin{document}
\maketitle

\begin{abstract}
Let $m,n\ge 2$.
For $m\times n$ complex matrices $A$, $C$ and an $n\times m$ matrix $B$,
Nobori (Linear Algebra Appl. 725 (2025) 135--144) proposed the following conjecture on the generalized B\"ottcher--Wenzel inequality:
$$
\norm{ABC-CBA}_F^2
\le 2\norm{B}_2^2\norm{A}_{(2),2}^2\norm{C}_F^2,
$$
where $\norm{\cdot}_F$ is the Frobenius norm, $\norm{\cdot}_2$ is the spectral norm, and $\norm{\cdot}_{(2),2}$ is the $(2,2)$-norm defined by $\norm{X}_{(2),2}=\sqrt{\sigma_1(X)^2+\sigma_2(X)^2}$,
in which $\sigma_1(X)$ and $\sigma_2(X)$ are the largest and second largest singular values of $X$, respectively.
This paper proves the conjecture and, as a consequence, establishes the Kronecker product inequality that Nobori derived from it.
In addition, we show that for each fixed $m,~n\ge 2$, the constant $2$ is sharp.
We also give necessary and sufficient conditions for equality, namely, the equality condition in the B\"ottcher--Wenzel inequality and some matrix identities derived from a fixed singular value decomposition.



\end{abstract}

\section{Introduction}
Let $M_{m,n}(\C)$ denote the vector space of all $m\times n$ complex matrices.
In particular, for $m=n$, we write it as $M_n(\C)$.
Let $I_n$ denote the identity matrix of order $n$, and let $\mathbf{0}$ denote the zero matrix of a proper size.
For $X\in M_{m,n}(\C)$, the conjugate transpose of $X$ is denoted by $X^*$.
Let $X\in M_{m,n}(\C)$. If $m\geq n$, then the nonnegative square roots of the eigenvalues of the positive semidefinite matrix $X^*X$ are called the singular values of $X$.
If $m<n$, then the nonnegative square roots of the eigenvalues of the positive semidefinite matrix $XX^*$ are called the singular values of $X$.
Let $\sigma_1(X)\geq\cdots\geq\sigma_r(X)$ denote the singular values of $X$ with $r=\min\{m, n\}$.

The vector space $M_{m,n}(\C)$ is equipped with the usual inner product $(X,Y)=\tr(XY^*)$ (where $\tr(Z)$ denotes the trace of $Z\in M_m(\C)$),
which satisfies
$
(X,Y)=\overline{(Y,X)}.
$
For $X\in M_{m,n}(\C)$, the Frobenius norm is given by $\norm{X}_F=\sqrt{(X,X)}$;
the spectral norm is defined by $\norm{X}_2=\sigma_1(X)$;
and the $(2,2)$-norm is defined by
$$
\norm{X}_{(2),2}=\sqrt{\sigma_1(X)^2+\sigma_2(X)^2}.
$$
When a matrix has at most one singular value, we set $\sigma_2(X)=0$.

For $A,B\in M_{n}(\C)$, the inequality
\begin{equation}
\norm{AB-BA}_F^2
\le 2\norm{A}_{F}^2\norm{B}_F^2
\label{eq:BW00}
\end{equation}
is given in \cite{Audenaert2010,Bottcher&Wenzel2008,Ge2020,Lu2011,Vong&Jin2008}.
Indeed, the upper bound in \eqref{eq:BW00} can be replaced by the following sharper bound \cite{Audenaert2010,Ge2020}:
\begin{equation}
\norm{AB-BA}_F^2
\le 2\norm{A}_{(2),2}^2\norm{B}_F^2.
\label{eq:BW}
\end{equation}
In particular, Cheng, Fong, and Lei~\cite{ChengFongLei2013} gave necessary and sufficient conditions for equality in \eqref{eq:BW}.
In recent years, attempts have been made to generalize the B\"{o}ttcher-Wenzel inequality \eqref{eq:BW00} by extending the ideas of commutators and norms \cite{Kimura&Ohno&Singal2023,Mayumi&Kimura&Ohno2024}. 
In a different direction, estimates have also been studied for the Frobenius norm of a matrix formed from the product and difference of three real square matrices \cite{Laszlo2022}.
Nobori \cite{Nobori2025} also proved the following result:
Let $A,C\in M_{m,n}(\C)$ and $B\in M_{n,m}(\C)$. Then
\begin{equation*}
\norm{ABC-CBA}_F^2
\le
2\norm{C}_2^2\norm{A}_{(2),2}^2\norm{B}_F^2.
\end{equation*}
When $m=1$ or $n=1$, Nobori gave a stronger estimate with constant $1$.
Moreover, Nobori proposed the following conjecture~\cite[Conjecture~3.1]{Nobori2025}.

\begin{conjecture}\label{conjNobori}
For all $A,C\in M_{m,n}(\C)$ and $B\in M_{n,m}(\C)$ with $m,n\geq2$,
\begin{equation}\label{eq:conj}
\norm{ABC-CBA}_F^2
\le
2\norm{B}_2^2\norm{A}_{(2),2}^2\norm{C}_F^2.
\end{equation}
\end{conjecture}
\noindent
In fact, Nobori \cite{Nobori2025} showed that \eqref{eq:conj} always holds when $m=1$ or $n=1$.
Nobori also proved \cite{Nobori2025} that Conjecture~\ref{conjNobori} implies the following inequality:
\begin{equation}\label{eqNobori01}
\norm{ABC-CBA}_F^2
\le
\norm{B}_2^2\norm{A\otimes C-C\otimes A}_F^2,
\end{equation}
where $X\otimes Y$ denotes the Kronecker product of $X$ and $Y$.

The paper is organized as follows. 
Section \ref{Preliminaries} gives some definitions and lemmas. 
Section \ref{sec:newproof} provides a complete proof of the conjecture, which also shows that \eqref{eqNobori01} holds, and demonstrates that $2$ is the optimal constant for any fixed $m,n\ge 2$. 
Section \ref{Equality} gives necessary and sufficient conditions for equality in \eqref{eq:conj}.

\section{Preliminaries}\label{Preliminaries}
In this section, we present the idea for proving Conjecture \ref{conjNobori} and give some concepts and lemmas needed for the proof. 
If one of the matrices $A$, $B$, and $C$ is the zero matrix, then \eqref{eq:conj} always holds, and equality holds as well.
Therefore, throughout this paper, we assume that none of $A$, $B$, and $C$ is the zero matrix.
Next, we rewrite \eqref{eq:conj} in an equivalent form.

\begin{lemma}
\label{le:eqconj31}
Let $m,n\ge2$, $A,C\in\M_{m,n}(\C)$, and $B\in\M_{n,m}(\C)$. Then \eqref{eq:conj} holds if and only if 
\begin{align}\label{eqconj31}
\norm{\Phi(D)}_F^2
\le
2\norm{A}_{(2),2}^2\norm{C}_F^2,
\end{align}
where $D=\frac{B}{\norm{B}_2}$, $\Phi(Z)=AZC-CZA$.
\end{lemma}

\begin{proof}
Since $B$ is not the zero matrix, $D$ is well defined. In fact, we have $B=\norm{B}_2D$. So
\begin{align*}
ABC-CBA
&=A(\norm{B}_2D)C-C(\norm{B}_2D)A\\
&=\norm{B}_2\Phi(D).
\end{align*}
Hence,
\begin{align*}
\norm{ABC-CBA}_F^2=\norm{B}_2^2\norm{\Phi(D)}_F^2.
\end{align*}
This proves the equivalence.
\end{proof}

We note that $\norm{D}_2=1$. 
Hence, instead of working with a general matrix $B$, we work with a contraction matrix $D$, whose standard definition is given below.

\begin{definition}\cite{HornJohnson1991}
A matrix $X\in M_{n,m}(\C)$ is a contraction if $\norm{X}_2\le1$.
\end{definition}

For square matrices, the fact that a contraction is a convex combination of unitary matrices is standard; see Zhang~\cite{Zhang2011}. 
However, a non-square contraction matrix can still be written as a convex combination of isometries or coisometries. 
We build on this through the following definitions and facts.

\begin{definition}\cite{Kubrusly2003}
A matrix $Q\in M_{n,m}(\mathbb C)$ is called an isometry if $Q^*Q=I_m$, and a coisometry if $QQ^*=I_n$.
\end{definition}

\begin{remark}
If $Q$ is an isometric matrix, then for any $X\in M_{m,n}(\mathbb C)$,
$$
(QX)^*(QX)=X^*X.
$$
Thus, $QX$ and $X$ have the same nonzero singular values, possibly with additional zero singular values. Hence,
\begin{align}\label{eq:semiunitary-tall}
\norm{QX}_F=\norm{X}_F,
~
\norm{QX}_{(2),2}=\norm{X}_{(2),2}.
\end{align}
Similarly, if $Q$ is a coisometric matrix, then
\begin{align}\label{eq:semiunitary-wide}
\norm{XQ}_F=\norm{X}_F,
~
\norm{XQ}_{(2),2}=\norm{X}_{(2),2}.
\end{align}
\end{remark}

\begin{lemma}
\label{lem:semiunitary-endpoint}
If $m,n\ge2$, $A,C\in\M_{m,n}(\C)$, and $Q\in M_{n,m}(\mathbb C)$ is an isometric or coisometric matrix, then
\begin{equation*}
\norm{AQC-CQA}_F^2
\le
2\norm{A}_{(2),2}^2\norm{C}_F^2.
\end{equation*}
\end{lemma}

\begin{proof}
If $Q$ is an isometric matrix, then by \eqref{eq:semiunitary-tall},
\begin{align*}
\norm{AQC-CQA}_F^2
&=\norm{Q(AQC-CQA)}_F^2\\
&=\norm{(QA)(QC)-(QC)(QA)}_F^2\\
&\le2\norm{QA}_{(2),2}^2\norm{QC}_F^2~(\text{by~\eqref{eq:BW}})\\
&=2\norm{A}_{(2),2}^2\norm{C}_F^2.
\end{align*}
If $Q$ is a coisometric matrix, then by \eqref{eq:semiunitary-wide}, we apply the same estimate to
$$
(AQC-CQA)Q=(AQ)(CQ)-(CQ)(AQ).
$$
\end{proof}

When $B$ is an isometry or a coisometry, \eqref{eq:conj} holds. 
This is important for proving \eqref{eq:conj}. 
However, to explicitly write a non-square contraction matrix as a convex combination of isometries or coisometries, several results are still needed. 
The first is the singular value decomposition \cite{zhan2013,Zhang2011}.

\begin{theorem}
Let $X\in M_{n,m}(\mathbb C)$ and $r=\min(m,n)$. Then there exist two orthonormal sets
\begin{align*}
  \{u_1,\ldots,u_r\}\subseteq\mathbb{C}^n,~
  \{v_1,\ldots,v_r\}\subseteq\mathbb{C}^m
\end{align*}
such that
$$
X=\sum_{i=1}^{r}\sigma_i(X)R_i, \text {~where~} R_i=u_iv_i^*.
$$
\end{theorem}

The following concepts and results are used to construct convex combinations and isometries (coisometries).

Let $\Omega_r$ be the Cartesian product of $r\geq1$ copies of the two-point set $\{-1,1\}$. 
Thus, $\Omega_r=\{-1,1\}^r$.
For any $\varepsilon=(\varepsilon_1,\ldots,\varepsilon_r)\in\Omega_r$, its $i$th component $\varepsilon_i$ is either $-1$ or $1$.

Let
\begin{equation}
c_i(\varepsilon_i)=\frac{1+\varepsilon_i s_i}{2},
\label{eq:local-sign-coefficient}
\end{equation}
where $i=1,\ldots,r$ and $s_1,\ldots,s_r\in[0,1]$.
Also, let
\begin{equation}
p_\varepsilon=\prod_{i=1}^r c_i(\varepsilon_i).
\label{eq:local-sign-coefficient01}
\end{equation}

\begin{remark}
\label{lem:local-sign-coefficients0}
Since $\varepsilon_i$ is either $-1$ or $1$,
$$
c_i(1)=\frac{1+s_i}{2}>0,
~
c_i(-1)=\frac{1-s_i}{2}\geq0.
$$
Thus, $c_i(\varepsilon_i)\geq0$. In particular, $c_i(\varepsilon_i)=0$ only when $s_i=1$ and $\varepsilon_i=-1$.
Also,
$$
c_i(1)+c_i(-1)=1,
~
c_i(1)-c_i(-1)=s_i.
$$
\end{remark}

\begin{remark}
\label{lem:local-sign-coefficients}
Since $c_i(\varepsilon_i)\geq0$, we have $p_\varepsilon\geq0$.
Moreover, $p_\varepsilon>0$ if and only if there is no index $i$ for which $s_i=1$ and $\varepsilon_i=-1$.
In particular, for the all-one sign vector
$$
\boldsymbol{1}=(1,\ldots,1)\in\Omega_r,
$$
we have
$p_{\boldsymbol{1}}>0.$
Also, if $s_j<1$, define
$\varepsilon^+:=\boldsymbol{1},$
and let $\varepsilon^-\in\Omega_r$ satisfy
$$
\varepsilon_j^-=-1,
~
\varepsilon_i^-=1
\quad(i\ne j).
$$
Then
$
p_{\varepsilon^+}>0,
~
p_{\varepsilon^-}>0.
$
\end{remark}

Next, we introduce other properties of $p_\varepsilon$.

\begin{proposition}
\label{prop:sign-weight-moments}
\begin{equation*}
\sum_{\varepsilon\in\Omega_r}p_\varepsilon=1.
\end{equation*}
Thus, $\{p_\varepsilon\}_{\varepsilon\in\Omega_r}$ are convex weights.
Moreover, for each fixed $j\in\{1,\ldots,r\}$,
\begin{equation*}
\sum_{\varepsilon\in\Omega_r}
p_\varepsilon\varepsilon_j=s_j.
\end{equation*}
\end{proposition}

\begin{proof}
Summing over $\Omega_r$ is the same as summing over each coordinate $\varepsilon_1,\ldots,\varepsilon_r$ separately. By the distributive law for finite sums,
\begin{align*}
\sum_{\varepsilon\in\Omega_r}p_\varepsilon
&=
\sum_{\varepsilon_1\in\{-1,1\}}
\cdots
\sum_{\varepsilon_r\in\{-1,1\}}
\prod_{i=1}^r c_i(\varepsilon_i)\\
&=
\sum_{\varepsilon_1\in\{-1,1\}}
\cdots
\sum_{\varepsilon_{r-1}\in\{-1,1\}}
\left(\left(\prod_{i=1}^{r-1} c_i(\varepsilon_i)\right)(c_r(1)+c_r(-1))\right)\\
&=
\prod_{i=1}^r
\bigl(c_i(1)+c_i(-1)\bigr)\\
&=1.~(\text{by Remark~\ref{lem:local-sign-coefficients0}})
\end{align*}

Now fix $j\in\{1,\ldots,r\}$. Then
\begin{align*}
\sum_{\varepsilon\in\Omega_r}p_\varepsilon\varepsilon_j
&=
\sum_{\varepsilon_1\in\{-1,1\}}
\cdots
\sum_{\varepsilon_r\in\{-1,1\}}
\varepsilon_jc_j(\varepsilon_j)\prod_{\substack{i=1\\i\ne j}}^r c_i(\varepsilon_i)\\
&=
\sum_{\varepsilon_1\in\{-1,1\}}
\cdots
\sum_{\varepsilon_{j-1}\in\{-1,1\}}
\sum_{\varepsilon_{j+1}\in\{-1,1\}}
\cdots
\sum_{\varepsilon_r\in\{-1,1\}}
\left(\bigl(c_j(1)-c_j(-1)\bigr)\prod_{\substack{i=1\\i\ne j}}^r c_i(\varepsilon_i)\right)\\
&=
\bigl(c_j(1)-c_j(-1)\bigr)
\prod_{\substack{i=1\\i\ne j}}^r
\bigl(c_i(1)+c_i(-1)\bigr)\\
&=s_j.~(\text{by Remark~\ref{lem:local-sign-coefficients0}})
\end{align*}
\end{proof}

With these preliminaries, we now establish the following results.

\begin{lemma}
\label{lem:convex-decomp}
Let $D\in M_{n,m}(\C)$ be a contraction and $r=\min\{m,n\}$.
If a singular value decomposition $\sum_{i=1}^r\sigma_i(D)R_i$ of $D$ is fixed, then
\begin{equation}
D
=
\sum_{\varepsilon\in\Omega_r}
p_\varepsilon Q_\varepsilon,
\label{eq:explicit-convex-decomposition}
\end{equation}
where
$p_\varepsilon
=
\prod_{i=1}^r\frac{1+\varepsilon_i\sigma_i(D)}{2}\geq0$,
$\varepsilon=(\varepsilon_1,\ldots,\varepsilon_r)\in\Omega_r$,
$\sum_{\varepsilon\in\Omega_r}p_\varepsilon=1$, and
\begin{equation}
Q_\varepsilon=
\sum_{i=1}^r\varepsilon_iR_i.
\label{eq:signed-endpoint}
\end{equation}
Moreover, \eqref{eq:signed-endpoint} has the following properties:
\begin{enumerate}[label=(\roman*)]
\item If $\sigma_j(D)<1$, and $\varepsilon^+$ and $\varepsilon^-$ are defined as in Remark~\ref{lem:local-sign-coefficients}, then
\begin{equation}
Q_{\varepsilon^+}-Q_{\varepsilon^-}
=
2R_j.
\label{eq:single-sign-flip}
\end{equation}

\item If $n\ge m$, then for every $\varepsilon\in\Omega_r$, $Q_\varepsilon$ is an isometric matrix.

\item If $n< m$, then for every $\varepsilon\in\Omega_r$, $Q_\varepsilon$ is a coisometric matrix.
\end{enumerate}
\end{lemma}

\begin{proof}
Since $D$ is a contraction, we have $0\le\sigma_i(D)\le1~(i=1,\ldots,r).$
Thus, by \eqref{eq:local-sign-coefficient} and \eqref{eq:local-sign-coefficient01}, $p_\varepsilon$ satisfies the required definition.
By Remark~\ref{lem:local-sign-coefficients} and Proposition~\ref{prop:sign-weight-moments}, 
$p_\varepsilon\ge0$ for every $\varepsilon\in\Omega_r$,
and $\sum_{\varepsilon\in\Omega_r}p_\varepsilon=1$.
We next prove \eqref{eq:explicit-convex-decomposition}. By \eqref{eq:signed-endpoint} and finite additivity,
\begin{align*}
\sum_{\varepsilon\in\Omega_r}
p_\varepsilon Q_\varepsilon
&=
\sum_{\varepsilon\in\Omega_r}
p_\varepsilon
\left(
\sum_{i=1}^r\varepsilon_iR_i
\right)\\
&=
\sum_{i=1}^r
\left(
\sum_{\varepsilon\in\Omega_r}
p_\varepsilon\varepsilon_i
\right)R_i\\
&=\sum_{i=1}^r\sigma_i(D)R_i
~(\text{by Proposition~\ref{prop:sign-weight-moments}})\\
&=D.
\end{align*}
We now prove the stated properties of $Q_\varepsilon$.
Since $\varepsilon^+$ and $\varepsilon^-$ differ only in the $j$th coordinate, \eqref{eq:signed-endpoint} gives
\begin{align*}
Q_{\varepsilon^+}-Q_{\varepsilon^-}
&=
\sum_{i=1}^r
\left(
\varepsilon_i^+-\varepsilon_i^-
\right)R_i\\
&=
\left(
\varepsilon_j^+-\varepsilon_j^-
\right)R_j\\
&=2R_j.
\end{align*}
Thus, \eqref{eq:single-sign-flip} holds.
By \eqref{eq:signed-endpoint},
$$
Q_\varepsilon
=
\sum_{i=1}^r\varepsilon_i u_iv_i^*.
$$
First assume $n\ge m$. Then $r=m$, $\{u_1,\ldots,u_r\}\subseteq\mathbb C^n$ is an orthonormal set, and
$$
v_1,\ldots,v_m
$$
form an orthonormal basis of $\C^m$.
Hence,
\begin{align*}
Q_\varepsilon^*Q_\varepsilon
&=
\left(
\sum_{i=1}^m\varepsilon_i v_iu_i^*
\right)
\left(
\sum_{j=1}^m\varepsilon_j u_jv_j^*
\right)\\
&=
\sum_{i,j=1}^m
\varepsilon_i\varepsilon_j
v_i(u_i^*u_j)v_j^*\\
&=
\sum_{i=1}^m
\varepsilon_i^2v_iv_i^*\\
&=I_m
~(\text{since }\varepsilon_i^2=1).
\end{align*}
Thus, $Q_\varepsilon$ is an isometric matrix.

Now assume $n< m$. The same argument shows that $Q_\varepsilon$ is a coisometric matrix. In particular, when $m=n$, every $Q_\varepsilon$ is a unitary matrix.
\end{proof}

The following result is important for proving that \eqref{eq:conj} holds and for characterizing the necessary and sufficient condition for equality in it.

\begin{lemma}
\label{lem:jensen}
Let $N$ be a positive integer, let $p_1,\ldots,p_N$ be nonnegative real numbers such that
$
\sum_{i=1}^N p_i=1,
$
and let $X_1,\ldots,X_N\in M_{m,n}(\C)$.
Then
\begin{equation}
\norm{\sum_{i=1}^N p_iX_i}_F^2
\le
\sum_{i=1}^N p_i\norm{X_i}_F^2.
\label{eq:jensen}
\end{equation}
Equality holds in \eqref{eq:jensen} if and only if
$
X_i=X_j
$
for any $i,j$ such that $p_i>0$ and $p_j>0$.
\end{lemma}

\begin{proof}
For any $i,j$,
\begin{align*}
\norm{X_i-X_j}_F^2
&=
\norm{X_i}_F^2+\norm{X_j}_F^2
-2\operatorname{Re}\ip{X_i}{X_j},
\end{align*}
where
$
\operatorname{Re}\ip{X}{Y}
=
\frac12
\left(
\ip{X}{Y}+\ip{Y}{X}
\right).
$
Hence,
\begin{align*}
&\frac12\sum_{i,j=1}^N
p_ip_j\norm{X_i-X_j}_F^2\\
&=
\frac12\sum_{i,j=1}^N
p_ip_j\norm{X_i}_F^2
+
\frac12\sum_{i,j=1}^N
p_ip_j\norm{X_j}_F^2
-
\sum_{i,j=1}^N
p_ip_j\operatorname{Re}\ip{X_i}{X_j}.
\end{align*}
For the first term, using $\sum_{j=1}^N p_j=1$, we obtain
\begin{align*}
\frac12\sum_{i,j=1}^N
p_ip_j\norm{X_i}_F^2
&=
\frac12
\sum_{i=1}^N
p_i\norm{X_i}_F^2
\sum_{j=1}^N p_j\\
&=
\frac12
\sum_{i=1}^N
p_i\norm{X_i}_F^2.
\end{align*}
Similarly, the second term satisfies
\begin{align*}
\frac12\sum_{i,j=1}^N
p_ip_j\norm{X_j}_F^2
&=
\frac12
\sum_{j=1}^N
p_j\norm{X_j}_F^2.
\end{align*}
Thus, the sum of the first two terms is
$$
\sum_{i=1}^N p_i\norm{X_i}_F^2.
$$
For the inner-product term, since $p_i$ and $p_j$ are real numbers,
\begin{align*}
\sum_{i,j=1}^N
p_ip_j\ip{X_i}{X_j}
&=
\ip{\sum_{i=1}^N p_iX_i}
{\sum_{j=1}^N p_jX_j}\\
&=
\norm{\sum_{i=1}^N p_iX_i}_F^2.
\end{align*}
Since $\norm{\sum_{i=1}^N p_iX_i}_F^2$ is real,
\begin{align*}
\sum_{i,j=1}^N
p_ip_j\operatorname{Re}\ip{X_i}{X_j}
&=
\operatorname{Re}
\sum_{i,j=1}^N
p_ip_j\ip{X_i}{X_j}\\
&=
\norm{\sum_{i=1}^N p_iX_i}_F^2.
\end{align*}
Therefore,
\begin{equation}
\begin{aligned}
\sum_{i=1}^N p_i\norm{X_i}_F^2-
\norm{\sum_{i=1}^N p_iX_i}_F^2
=
\frac12\sum_{i,j=1}^N
p_ip_j\norm{X_i-X_j}_F^2.
\label{eq:jensen-gap}
\end{aligned}
\end{equation}
Every term on the right-hand side of \eqref{eq:jensen-gap} is nonnegative. Therefore, \eqref{eq:jensen} holds.

Finally, we discuss the equality condition. By \eqref{eq:jensen-gap}, equality holds in \eqref{eq:jensen} if and only if
$$
\frac12\sum_{i,j=1}^N
p_ip_j\norm{X_i-X_j}_F^2=0.
$$
This holds exactly when
$$
p_ip_j\norm{X_i-X_j}_F^2=0,
$$
which in turn holds precisely when $X_i=X_j$ for any $p_i>0$ and $p_j>0$.
\end{proof}

\section{Proof of Conjecture~\ref{conjNobori}}\label{sec:newproof}
In this section, we first give a complete proof of Conjecture~\ref{conjNobori}, and then show that the constant $2$ in \eqref{eq:conj} is optimal for every fixed $m,n\ge2$.

\begin{theorem}
\label{thm:conj31}
Let $m,n\ge2$, $A,C\in\M_{m,n}(\C)$, and $B\in\M_{n,m}(\C)$. Then \eqref{eq:conj} holds.
\end{theorem}

\begin{proof}
By Lemma \ref{le:eqconj31}, it suffices to prove that \eqref{eqconj31} holds.
First assume $n\ge m$. 
By Lemma~\ref{lem:convex-decomp}, there exist $p_\varepsilon\ge0$ and isometric matrices $Q_\varepsilon$ such that
$$
D=\sum_\varepsilon p_\varepsilon Q_\varepsilon,
~
\sum_\varepsilon p_\varepsilon=1,
~
Q_\varepsilon^*Q_\varepsilon=I_m.
$$
By the linearity of $\Phi$,
$$
\Phi(D)=\sum_\varepsilon p_\varepsilon\Phi(Q_\varepsilon),
$$
where $\Phi(Q_\varepsilon)=AQ_\varepsilon C-CQ_\varepsilon A$.
By Lemma~\ref{lem:jensen},
\begin{align*}
\norm{\Phi(D)}_F^2
&\le
\sum_\varepsilon p_\varepsilon
\norm{\Phi(Q_\varepsilon)}_F^2.
\end{align*}
By Lemma~\ref{lem:semiunitary-endpoint}, for every $\varepsilon$,
\begin{equation}\label{eq:each-endpoint-bound}
\norm{\Phi(Q_\varepsilon)}_F^2
\le2\norm{A}_{(2),2}^2\norm{C}_F^2.
\end{equation}
Therefore,
\begin{align*}
\norm{\Phi(D)}_F^2
&\le
\sum_\varepsilon p_\varepsilon
\left(2\norm{A}_{(2),2}^2\norm{C}_F^2\right)\\
&=2\norm{A}_{(2),2}^2\norm{C}_F^2.
\end{align*}
This proves the case $n\ge m$.
If $m>n$, Lemma~\ref{lem:convex-decomp} writes $D$ as a convex combination of coisometric matrices. 
We then use the coisometric case of Lemma~\ref{lem:semiunitary-endpoint}; 
the remaining steps are the same. 
Thus, \eqref{eq:conj} holds.
\end{proof}

\begin{remark}
Since Theorem~\ref{thm:conj31} proves Conjecture~\ref{conjNobori}, the implication established by Nobori immediately yields \eqref{eqNobori01}.
\end{remark}

Next, we show that the constant $2$ is optimal.

\begin{proposition}
For every fixed $m,n\ge2$, the constant $2$ in \eqref{eq:conj} is best possible.
\end{proposition}

\begin{proof}
It is enough to give, for each fixed $(m,n)$, an example for which the constant $2$ is attained.
Let
$$
A_2=
\begin{bmatrix}
1&0\\
0&-1
\end{bmatrix},
~
C_2=
\begin{bmatrix}
0&1\\
0&0
\end{bmatrix}.
$$
A direct calculation gives
$$
A_2C_2-C_2A_2
=
\begin{bmatrix}
0&2\\
0&0
\end{bmatrix},
$$
so
\begin{equation}
\norm{A_2C_2-C_2A_2}_F^2=4.
\label{eq:2x2-left}
\end{equation}
Also,
$$
\norm{A_2}_{(2),2}^2=1^2+1^2=2,
~
\norm{C_2}_F^2=1.
$$

\textbf{Case (a): $n\ge m$.} Let
$$
B=
\begin{bmatrix}
I_m\\
\mathbf{0}_{(n-m)\times m}
\end{bmatrix}
\in\M_{n,m}(\C).
$$
Clearly,
$$
B^*B=I_m,
~
\norm{B}_2=1.
$$
Construct the $m\times m$ matrices
$$
\widehat A=\begin{bmatrix}
A_2&\mathbf{0}\\
\mathbf{0}&\mathbf{0}_{m-2}
\end{bmatrix}
~
\widehat C=\begin{bmatrix}
C_2&\mathbf{0}\\
\mathbf{0}&\mathbf{0}_{m-2}
\end{bmatrix},
$$
and let
$$
A=
\begin{bmatrix}
\widehat A&\mathbf{0}_{m\times(n-m)}
\end{bmatrix},
~
C=
\begin{bmatrix}
\widehat C&\mathbf{0}_{m\times(n-m)}
\end{bmatrix}.
$$
Then
\begin{align*}
ABC-CBA
&=
\begin{bmatrix}
\widehat A\widehat C-\widehat C\widehat A
&\mathbf{0}
\end{bmatrix}.
\end{align*}
Its only possibly nonzero upper-left $2\times2$ block is $A_2C_2-C_2A_2$. Hence, by \eqref{eq:2x2-left},
$$
\norm{ABC-CBA}_F^2=4.
$$
On the other hand, adding zero columns to the right does not change the nonzero singular values. Therefore,
$$
\norm{A}_{(2),2}^2=2,
~
\norm{C}_F^2=1.
$$
Thus, equality holds.

\textbf{Case (b): $n< m$.} Let
$$
B=
\begin{bmatrix}
I_n&\mathbf{0}_{n\times(m-n)}
\end{bmatrix}
\in\M_{n,m}(\C).
$$
Then
$$
BB^*=I_n,
~
\norm{B}_2=1.
$$
Construct
$$
\widetilde A=\begin{bmatrix}
A_2&\mathbf{0}\\
\mathbf{0}&\mathbf{0}_{n-2}
\end{bmatrix},
~
\widetilde C=\begin{bmatrix}
C_2&\mathbf{0}\\
\mathbf{0}&\mathbf{0}_{n-2}
\end{bmatrix},
$$
and let
$$
A=
\begin{bmatrix}
\widetilde A\\
\mathbf{0}_{(m-n)\times n}
\end{bmatrix},
~
C=
\begin{bmatrix}
\widetilde C\\
\mathbf{0}_{(m-n)\times n}
\end{bmatrix}.
$$
Then
$$
ABC-CBA
=
\begin{bmatrix}
\widetilde A\widetilde C-\widetilde C\widetilde A\\
\mathbf{0}
\end{bmatrix}.
$$
As in the previous case, we obtain
$$
\norm{ABC-CBA}_F^2=4,
~
\norm{A}_{(2),2}^2=2,
~
\norm{C}_F^2=1,
~
\norm{B}_2^2=1.
$$
Thus, equality holds.
\end{proof}

\section{Necessary and Sufficient Conditions for Equality}\label{Equality}
We now give an equality criterion for \eqref{eq:conj}.

\begin{theorem}
Let $D=\frac{B}{\norm{B}_2}\in M_{n,m}(\C)$ and $r=\min\{m,n\}$.
If a singular value decomposition $\sum_{i=1}^r\sigma_i(D)R_i$ of $D$ is fixed, then equality holds in \eqref{eq:conj} if and only if the following two conditions hold:

(a) For every index $i$ such that $\sigma_i(D)<1$,
\begin{equation*}
AR_iC=CR_iA.
\end{equation*}

(b) Let $Q_{\boldsymbol{1}}=\sum_{i=1}^rR_i$
be constructed from the all-one vector $\boldsymbol{1}$ as in Lemma~\ref{lem:convex-decomp}.
If $n\ge m$, then
\begin{equation*}
\norm{(Q_{\boldsymbol{1}}A)(Q_{\boldsymbol{1}}C)-(Q_{\boldsymbol{1}}C)(Q_{\boldsymbol{1}}A)}_F^2
=
2\norm{A}_{(2),2}^2\norm{C}_F^2.
\end{equation*}
If $m> n$, then
\begin{equation*}
\norm{(AQ_{\boldsymbol{1}})(CQ_{\boldsymbol{1}})-(CQ_{\boldsymbol{1}})(AQ_{\boldsymbol{1}})}_F^2
=
2\norm{A}_{(2),2}^2\norm{C}_F^2.
\end{equation*}
\end{theorem}

\begin{proof}
If $n\geq m$, then
$Q_\varepsilon^*Q_\varepsilon=I_m$, while if $m>n$, then
$Q_\varepsilon Q_\varepsilon^*=I_n$. Therefore,
Lemma~\ref{lem:semiunitary-endpoint} gives \eqref{eq:each-endpoint-bound} in both cases. 
By Lemma \ref{le:eqconj31}, equality in \eqref{eq:conj} is equivalent to
\begin{equation}
\norm{\Phi(D)}_F^2=2\norm{A}_{(2),2}^2\norm{C}_F^2.
\label{eq:normalized-equality}
\end{equation}
By Lemma~\ref{lem:convex-decomp},
\begin{equation*}
D
=
\sum_{\varepsilon\in\Omega_r}
p_\varepsilon Q_\varepsilon,
~
p_\varepsilon\ge0,
~
\sum_{\varepsilon\in\Omega_r}p_\varepsilon=1.
\end{equation*}
By the proof of Theorem \ref{thm:conj31},
\begin{equation}\label{eq:equality-chain}
\begin{aligned}
\norm{\Phi(D)}_F^2
&\le
\sum_{\varepsilon\in\Omega_r}
p_\varepsilon
\norm{\Phi(Q_\varepsilon)}_F^2\\
&\le
2\norm{A}_{(2),2}^2\norm{C}_F^2.
\end{aligned}
\end{equation}
We now prove necessity and sufficiency separately.

Assume that equality holds in \eqref{eq:conj}. By \eqref{eq:normalized-equality} and \eqref{eq:equality-chain},
\begin{equation}
\sum_{\varepsilon\in\Omega_r}
p_\varepsilon\norm{\Phi(Q_\varepsilon)}_F^2
=2\norm{A}_{(2),2}^2\norm{C}_F^2.
\label{eq:average-endpoint-equality}
\end{equation}
By \eqref{eq:each-endpoint-bound}, for every $\varepsilon\in\Omega_r$,
$$
2\norm{A}_{(2),2}^2\norm{C}_F^2-\norm{\Phi(Q_\varepsilon)}_F^2\ge0.
$$
Together with \eqref{eq:average-endpoint-equality}, this gives
\begin{align*}
0
&=
2\norm{A}_{(2),2}^2\norm{C}_F^2-
\sum_{\varepsilon\in\Omega_r}
p_\varepsilon\norm{\Phi(Q_\varepsilon)}_F^2\\
&=
\sum_{\varepsilon\in\Omega_r}
p_\varepsilon
\left(
2\norm{A}_{(2),2}^2\norm{C}_F^2-
\norm{\Phi(Q_\varepsilon)}_F^2
\right).
\end{align*}
The right-hand side is a finite sum of nonnegative numbers. Therefore, whenever $p_\varepsilon>0$,
\begin{equation}
\norm{\Phi(Q_\varepsilon)}_F^2=2\norm{A}_{(2),2}^2\norm{C}_F^2.
\label{eq:positive-endpoint-max}
\end{equation}
Also, the first inequality in \eqref{eq:equality-chain} must be an equality. By the equality condition in Lemma~\ref{lem:jensen},
\begin{equation}
\Phi(Q_\varepsilon)=\Phi(Q_\delta)
~
\text{whenever }
p_\varepsilon>0
\text{ and }
p_\delta>0.
\label{eq:all-images-equal}
\end{equation}
Now fix an index $i$ such that $\sigma_i(D)<1$. By Lemma~\ref{lem:convex-decomp}, there are two sign vectors
$
\varepsilon^+,\varepsilon^-\in\Omega_r
$
such that
$
p_{\varepsilon^+}>0,
~
p_{\varepsilon^-}>0,
$
and
\begin{equation}
Q_{\varepsilon^+}-Q_{\varepsilon^-}=2R_i.
\label{eq:equality-single-flip}
\end{equation}
By \eqref{eq:all-images-equal},
$$
\Phi(Q_{\varepsilon^+})
=
\Phi(Q_{\varepsilon^-}).
$$
Using the linearity of $\Phi$ and \eqref{eq:equality-single-flip},
\begin{align*}
0
&=
\Phi(Q_{\varepsilon^+})
-\Phi(Q_{\varepsilon^-})\\
&=
\Phi(Q_{\varepsilon^+}-Q_{\varepsilon^-})\\
&=
\Phi(2R_i).
\end{align*}
This proves condition (a).
We next prove condition (b).
By Lemma~\ref{lem:convex-decomp},
$
p_{\boldsymbol{1}}>0.
$
Hence, by \eqref{eq:positive-endpoint-max},
\begin{equation}
\norm{\Phi(Q_{\boldsymbol{1}})}_F^2=2\norm{A}_{(2),2}^2\norm{C}_F^2.
\label{eq:Q0-Phi-max}
\end{equation}
Thus, condition (b) holds, by \eqref{eq:semiunitary-tall} and \eqref{eq:semiunitary-wide}.

Now assume that conditions (a) and (b) both hold. First, \eqref{eq:Q0-Phi-max} holds.
Also,
\begin{align*}
\Phi(D)-\Phi(Q_{\boldsymbol{1}})
&=
\Phi\left(
\sum_{i=1}^r \sigma_i(D)R_i
-
\sum_{i=1}^rR_i
\right)\\
&=
\sum_{i=1}^r\bigl(\sigma_i(D)-1\bigr)\Phi(R_i).
\end{align*}
If $\sigma_i(D)=1$, then
$
\bigl(\sigma_i(D)-1\bigr)\Phi(R_i)=0.
$
If $\sigma_i(D)<1$, then by condition (a),
$$
\Phi(R_i)=AR_iC-CR_iA=0,
$$
so again
$
\bigl(\sigma_i(D)-1\bigr)\Phi(R_i)=0.
$
Therefore,
$
\Phi(D)=\Phi(Q_{\boldsymbol{1}}).
$
Hence, \eqref{eq:normalized-equality} holds, which proves sufficiency.
\end{proof}
\begin{remark}
Condition (b) is precisely the equality condition for B\"ottcher--Wenzel inequality \eqref{eq:BW}, as characterized in \cite{ChengFongLei2013}.
In particular, when all singular values of $D$ are $1$, condition (a) is absent. 
In this case, the equality condition reduces to that of the B\"ottcher--Wenzel inequality \cite{ChengFongLei2013}.
\end{remark}

~\\

\end{document}